\documentclass[12pt]{amsart}
\usepackage{amsfonts}
\usepackage{amssymb}
\usepackage{url}
\usepackage{tikz}
\usepackage{graphicx}
\usepackage[letterpaper, left=2.5cm, right=2.5cm, top=2.5cm,
bottom=2.5cm,dvips]{geometry}
\usepackage[normalem]{ulem}

\newtheorem{theorem}{Theorem}
\theoremstyle{plain}

\newtheorem{claim}[theorem]{Claim}

\newtheorem{corollary}[theorem]{Corollary}

\newtheorem{definition}[theorem]{Definition}

\newtheorem{lemma}[theorem]{Lemma}

\newtheorem{proposition}[theorem]{Proposition}

\newenvironment{claimproof}{\noindent{\em Proof of the claim.}}{\qedclaim}
\newcommand{\qedclaim}{\hfill $\diamond$ \medskip}

\numberwithin{equation}{section}

\begin{document}
\title[Centroidal localization game]{Centroidal localization game}
\author[B. Bosek]{Bart\l omiej Bosek}
\address{Theoretical Computer Science Department, Faculty of Mathematics and Computer Science, Jagiellonian
University, 30-348 Krak\'{o}w, Poland}
\email{bosek@tcs.uj.edu.pl}

\author[P. Gordinowicz] {Przemys\l aw Gordinowicz}
\address{Institute of Mathematics, Lodz University of Technology, \L{}\'od\'z, Poland}
\email{pgordin@p.lodz.pl}

\author[J. Grytczuk]{Jaros\l aw Grytczuk}
\address{Faculty of Mathematics and Information Science, Warsaw University
	of Technology, 00-662 Warsaw, Poland}
\email{j.grytczuk@mini.pw.edu.pl}

\author[N; Nisse]{Nicolas Nisse}
\address{Universit\'e C\^ote d'Azur, Inria, CNRS, I3S, France}
\email{nicolas.nisse@inria.fr}

\author[J. Sok\'{o}\l ]{Joanna Sok\'{o}\l }
\address{Faculty of Mathematics and Information Science, Warsaw University
	of Technology, 00-662 Warsaw, Poland}
\email{j.sokol@mini.pw.edu.pl}
\author[M. \'{S}leszy\'{n}ska-Nowak]{Ma\l gorzata \'{S}leszy\'{n}ska-Nowak}

\address{Faculty of Mathematics and Information Science, Warsaw University
	of Technology, 00-662 Warsaw, Poland}
\email{m.sleszynska@mini.pw.edu.pl}

\thanks{This work was supported by The National Center for Research and
	Development under the project PBS2/B3/24/2014, by the ANR project Stint (ANR-13-BS02-0007) (N. Nisse) and the associated Inria team AlDyNet (N. Nisse)}

\begin{abstract}
One important problem in a network $G$ is to locate an (invisible) moving entity by using distance-detectors placed at strategical locations in $G$. For instance, the famous {\it metric dimension} of a graph $G$ is the minimum number $k$ of detectors placed in some vertices $\{v_1,\cdots,v_k\}$ such that the vector $(d_1,\cdots,d_k)$ of the distances $d(v_i,r)$ between the detectors and the entity's location $r$ allows to uniquely determine $r$ for every $r \in V(G)$. In a more realistic setting, each device does not get the exact distance to the entity's location. Rather, given locating devices placed in $\{v_1,\cdots,v_k\}$, we get only relative distances between the moving entity's location $r$ and the devices (roughly, for every $1\leq i,j\leq k$, it is provided whether $d(v_i,r) >$, $<$, or $=$ to $d(v_j,r)$). The {\it centroidal dimension} of a graph $G$ is the minimum number of devices required to locate the entity, in one step, in this setting.

In this paper, we consider the natural generalization of the latter problem, where vertices may be probed sequentially (i.e., in several steps) until the moving entity is located. Roughly, at every turn, a set $\{v_1,\cdots,v_k\}$ of vertices are probed and then the relative order of the distances between the vertices $v_i$ and the current location $r$ of the moving entity is given. If it not located, the moving entity may move along one edge. Let $\zeta^* (G)$ be the minimum $k$ such that the entity is eventually located, whatever it does, in the graph $G$. 

We first prove that $\zeta^* (T)\leq 2$ for every tree $T$ and give an upper bound on $\zeta^*(G\square H)$ for the cartesian product of graphs $G$ and $H$. Our main result is that $\zeta^* (G)\leq 3$ for any outerplanar graph $G$. We then prove that $\zeta^* (G)$ is bounded by the pathwidth of $G$ plus 1 and that the optimization problem of determining $\zeta^* (G)$ is NP-hard in general graphs.
Finally, we show that approximating (up to a small constant distance) the location of the robber in the Euclidean plane requires at most two vertices per turn.

\end{abstract}

\maketitle

\section{Introduction}

The problem of locating or capturing an intruder in a graph has been widely studied using many different approaches. One approach is to place detection devices at some vertices of the graph such that these devices precisely determine the position of the intruder at any moment and wherever it is. For instance, this is the approach taken by identifying codes, metric bases and centroidal bases.

Recall that a {\it dominating set} $D \subseteq V$ of a graph $G=(V,E)$ is a set such that any vertex of $V$ is in the closed neighborhood of some vertex of $D$.  That is, $V= \bigcup_{v \in D} N[v]$, where $N(v)=\{u \in V \mid \{u,v\} \in E\}$ and $N[v]=N(v) \cup \{v\}$ for any $v\in V$. A vertex $u$ {\it separates} two vertices $v$ and $w$ if it is in the closed neighborhood of exactly one of them. A set $D \subseteq V$ separates the vertices of a set $X$ if, for every two vertices $v,w\in X$, there exists $u\in D$ which separates them. A set $D$ is an \emph{identifying code} of a graph $G$ if it is dominating and separates all vertices from $V(G)$~\cite{Karp}. Similarly, Slater defined the notion of \emph{locating-dominating} set, that is a dominating set $D$ separating vertices of $V(G) \setminus D$~\cite{Slater1,Slater2} (see~\cite{Lobstein} and the references therein). Both previous approaches model the situation when the detection devices can detect an intruder at distance at most one from them.

The case when the devices have a longer range of detection has also been considered. For instance, Slater considered the case of an infinite range of detection. Precisely, a  \emph{locating set} is a vertex set $L=\{w_1,...,w_k\}\subseteq V(G)$ such that, for each vertex $v\in V(G)$, the ordered $k$-tuple $(d(v,w_1),d(v,w_2),...,d(v,w_k))$ of distances between the detectors and the intruder vertex $v$ is unique~\cite{Slater,Harary} (i.e., the vector of distances allows to determine uniquely the vertex $v$). The minimum cardinality of a locating set is called \emph{metric dimension} of a graph, denoted by $\operatorname{MD}(G)$, and a locating set of minimal cardinality is called a \emph{metric basis} of $G$. For more details on the complexity of computing the metric dimension and on bounds on metric dimension of various graph classes, see for instance~\cite{FMR+17a,FMR+17b}. 

The concept of {\it centroidal bases} (see eg.~\cite{Foucaud}) is similar to the one of metric bases. Once again, detecting devices placed at the vertices $C=\{c_{1},\ldots, c_{k}\}\subseteq V(G)$ are assume to have have unlimited range, but they do not determine the exact distance from the intruder. Instead, the devices report in an order: if $c_i$ is closer to the intruder than $c_j$, then $c_i$ reports before $c_j$ and, if they are at the same distance from the intruder, then they report simultaneously. In other words, the received information is an ordered partition of $C$ (ordered by nondecreasing distance from the intruder and ties noted). If, for any intruder position $v\in V(G)$, the received information allows us to uniquely determine $v$, then $C$ is called a \emph{centroidal locating set}. If $C$ is a centroidal locating set of minimal cardinality, then it is said to be a \emph{centroidal basis} of $G$, and its cardinality is called the \emph{centroidal dimension} of $G$, denoted by $\operatorname{CD}(G)$. For instance, it is known that, for any $n$-node graph $G$ with maximum degree at least $2$, then $(1+o(1)) \frac{\ln n}{\ln \ln n} \leq \operatorname{CD}(G) \leq n-1$~\cite{Foucaud}. We note that additional information, whether for some checked vertex the distance from the intruder is 0, allows to remove the degree restriction.

All above mentioned models aim at locating the intruder at any moment of time or, equivalently, at any turn. Another approach consists in locating it in a finite number of turns. This is in the vein of the famous Cops and Robber games where a team of cops must capture a (generally) visible robber by moving alternately in a graph (e.g., see the book~\cite{BonatoNowakowski}).

The {\it localization game}~\cite{Seager1,LocalizationGame,Haslegrave} somehow generalizes the notion of metric dimension of a graph by allowing to probe several sets (of bounded size) instead of only one. Of course, between any two probes, the intruder may move (since otherwise, it would be sufficient to check every vertex one by one).  Precisely, the localization game is defined as follows. Let $G=(V,E)$ be a simple undirected graph and let $k\geq 1$ be a
fixed integer. Two players, the {\it Cop-player} and the {\it Robber-player} (the {\it robber}), play alternately as follows. In the first turn, the robber chooses a vertex $r \in V$ but keeps it in a secret. Then, at every turn, first the Cop-player picks (or probes) $k$ vertices $B=\{v_{1},v_{2},\ldots ,v_{k}\} \in V^k$ and, in return, gets the vector $D(B)=(d_{1},d_{2},\ldots ,d_{k})$ where $d_{i}=d_{G}(r,v_{i})$ is the distance (in $G$) from $r$ to $v_i$ for every $i=1,2,\ldots ,k$. If the location of the robber is uniquely identified thanks to this information, the game ends with the victory of the Cop-player. Otherwise, the robber may move along one edge. The robber wins if its location is never known. Let the \emph{localization number} of $G$, denoted by $\zeta (G)$, be the least integer $k$ for which the cops have a winning strategy whatever be the strategy of the robber.

This game restricted to $k=1$ has been introduced by Seager \cite{Seager1}, and studied further in \cite{Brandt,WestTCS,Seager2}. The parameter $\zeta (G)$ can be seen as the game theoretic variant of $\operatorname{MD}(G)$ and, by the definition, $\zeta(G) \leq \operatorname{MD}(G)$ in any graph $G$.
The localization game with many cops has been introduced recently by the authors of this paper in~\cite{LocalizationGame} and, independently, in~\cite{Haslegrave}. For instance, in~\cite{Haslegrave}, it was shown that $\zeta(G) \leq \lfloor \frac{(\Delta+1)^2}{4}\rfloor+1$ for any graph $G$ with maximum degree $\Delta$. The main result in~\cite{LocalizationGame} is that $\zeta(G)$ is unbounded in the class of planar graphs (more precisely, in the class of graphs obtained from a tree by adding a universal vertex). Moreover, computing $\zeta(G)$ is NP-hard~\cite{LocalizationGame}.

\medskip
The goal of this paper is to propose and study a ``generalization" of the centroidal dimension, in the same way as the localization game somehow extends the notion of metric dimension. This is inspired by localization problems in wireless networks (such as the network of Wi-Fi access points). We are interested in locating a person with a mobile device, who may walk along the network changing his position in time. The strength of the signal from the mobile device is proportional to its distance to particular access points. Unfortunately the signal may be easily disturbed by various factors, hence its strength depends highly on the circumstances. On the other hand, the relative order of the strengths of two signals is expected to be invariant regardless of the circumstances.

The \emph{centroidal localization game} is a turn-by-turn 2-Player game that proceeds as follows. First an invisible robber is placed at some vertex $r$. Then, at every turn, first the Cop-player probes a set $\{v_1,\cdots,v_k\}$ of $k$ vertices.  In return, the Cop-player receives, for any $1\leq i<j\leq k$, the information whether $d(v_i, r) = 0$ or $d(v_i,r)=d(r,v_j)$ or $d(v_i,r)<d(r,v_j)$ or $d(v_i,r)>d(r,v_j)$. If the location of the robber is uniquely identified thanks to this information, the game ends with the victory of the Cop-player. Otherwise, the robber may move along one edge. The robber wins if its location is never known. Note that it is not necessary for the Cop-player to probe the vertex occupied by the robber to win. 

The {\it centroidal localization number} of $G$, denoted by $\zeta^*(G)$, is the minimum $k$ that ensures the victory for the Cop-player whatever be the strategy of the robber. Note that $\zeta^*(G)$ may be viewed as game-theoretical version of the centroidal dimension, and the inequality $\zeta^*(G) \le \operatorname{CD}(G)$ obviously holds for every graph $G$. Note also that $\zeta(G) \le \zeta^*(G)$ holds by definition. Hence, by~\cite{LocalizationGame}, $\zeta^*(G)$ is unbounded on planar graphs (even in the class of graphs obtained from a tree by adding a universal vertex), while any upper bound proven in this paper holds for $\zeta(G)$ as well. It is also interesting to note that, some results obtained in the context of the localization game, have been proved without using the exact distances but only their relative order. For this reason, from the proof of~\cite{Haslegrave}, it is possible to directly derive that
\begin{corollary} \label{thm:Delta}
For any graph $G$ with maximum degree $\Delta$,  
$\zeta^*(G) \leq \lfloor \frac{(\Delta+1)^2}{4}\rfloor+1.$ 
\end{corollary}
 
\smallskip

\noindent {\bf Our results.} As a warm-up, we give easy results on the centroidal localization number (Section~\ref{sec:warmup}). Then, we show that $\zeta^*(T) \le 2$ for any tree $T$ (Section~\ref{sec:tree}) and we provide an upper bound on $\zeta^*(G\square H)$  (where $\square$ denotes the cartesian product) in Section~\ref{sec:product}. 
Our main result is that 
 $\zeta^* (G) \leq 3$ for any outerplanar graph $G$ (Section~\ref{sec:outer}), which also gives the best known bound for $\zeta(G)$ in this class of graphs. Then we show that deciding whether $\zeta^* (G) \leq k$ is NP-hard in the class of graphs $G$ with diameter $2$ (Section~\ref{sec:complexity}). Finally, we show that approximating (up to a small constant distance) the location of the robber in the Euclidean plane requires to probe at most two vertices per turns (Section~\ref{sec:euclidean}). In the final section (Section~\ref{sec:conclusion}) we set several open problems for future research.
 
\section{Warm-up}\label{sec:warmup}

Let us start with very simple observations. 

\begin{proposition}
$\zeta^*(G)=1$ if and only if $G$ is a graph with at most one edge.
\end{proposition}
\begin{proof}
First let $G$ be a graph with at most one edge. Then $G$ is a collection of isolated vertices and possibly one pair of vertices connected by the edge $e=\{u,v\}$. Note that if the robber starts the game in one of the isolated vertices, then he will remain in the same location throughout the whole game. The strategy for the Cop-player is to first check all isolated vertices one by one. If by doing so she doesn't catch the robber then he is in $u$ or $v$. Then the Cop-player checks $u$ and either immediately catches the robber, or knows that he is in $v$.

It is easy to see that if $G$ has more than 1 edge, then the robber should choose one of the edges and move between between its vertices. the Cop-player can locate the robber only if she checks the exact vertex with the robber, which can always be avoided.
\end{proof}

\begin{proposition} \label{prop:component}
For any graph $G$, $\zeta^*(G) = \max \{\zeta^*(C) \mid C \textrm{ is a connected component of } $G$\}$ or $\zeta^*(G)=2$ if $|E(G)|>1$ and every connected component of $G$ has at most one edge. 
\end{proposition}
\begin{proof}
The result is obvious if every connected component of $G$ has at most one edge. Otherwise, by the above Proposition, we may assume that the Cop-player can probe at least two vertices per turn. We show that, by checking two vertices in each round the Cop-player can determine the component of $G$ containing the robber.

Assume $G$ has more than one component. First note that the robber will never leave the component of the vertex of his first location. Moreover only vertices from this component are at finite distance from him. Say the Cop-player chooses vertices $c_1,\ c_2$, each from different components. Then at most one of $c_1,\ c_2$ are at finite distance from the robber. If their distances from the robber are equal, then both are from different component than the robber. Otherwise the component of vertex with smaller distance contains the robber. Once such a component $C$ has been detected, it is then sufficient to probe at most $\zeta^*(C)$ vertices per turn in $C$ to locate the robber. 
\end{proof}

\begin{proposition}
Every bipartite graph $G$ with partition classes of size $a$ and $b$	satisfies $\zeta^* (G)\leq \max (2, \min (a,b))$.
\end{proposition}
\begin{proof}
As the Cop-player can probe at each round at least 2 vertices, by Proposition~\ref{prop:component} we may assume that $G$ is connected. Assume $a \le b$. 

When $a = 1$ the Cop-player may use the following strategy. Check at each round the vertex of the smaller partition class, which forces the robber not to move (otherwise he will be caught immediately). Second probed vertex is chosen from another partition class checking the whole class in a~sequence. Eventually exact location of the robber will be checked.

When $a > 1$ at first round the Cop-player probes all $a$ vertices from the smaller partition class. Either the robber is caught or there is some vertex, say $v$, that is adjacent with the robber's position and hence it minimizes the distance. 
In the latter case, the robber must occupy a vertex in the neighborhood of $v$. In further rounds of the game the Cop-player probes $a-1$ vertices from the smaller partition class except $v$ and step by step one vertex, say $u_i$, from $N(v)$. Note, that when the robber moves, he can either go to some probed vertex (he is then caught immediately) or to vertex $v$. But then $d(u_i, v) = 1$, while other distances are at least 2, hence the robber is located at $v$. Therefore, the robber is forced to stay in a vertex chosen at the first round. Eventually exact location of the robber will be checked.   
\end{proof}

It is easy to note, that for a path $P_n$ for $n \ge 3$ there is $\zeta^*(P_n) = 2$. Indeed, the Cop-player wins sweeping the path from one of its ends probing at each step 2 adjacent vertices. Since paths are precisely the graphs with pathwidth one, this remark can be generalized as follows.

A {\it path-decomposition} of a graph $G=(V,E)$ is a sequence ${\mathcal X}=(X_1,\cdots,X_t)$ of subsets of $V$, called {\it bags}, such that, for every edge $\{u,v\} \in E$, there exists a	bag containing both $u$ and $v$, and such that, for every $1\leq i \leq k \leq j \leq t$, $X_i \cap X_j \subseteq X_k$. The {\it width} of $\mathcal X$ equals $\max_{1\leq i \leq t} |X_i|-1$ and the {\it pathwidth}  of $G$, denoted by $\operatorname{pw}(G)$, is the minimum width of its path-decompositions. Pathwidth and path-decompositions are closely related to some kind of pursuit-evasion games~\cite{Bienstock}.

\begin{proposition}
Every graph $G$ satisfies $\zeta^*(G) \leq \operatorname{pw}(G)+1$. \end{proposition}

\begin{proof}
	By Proposition~\ref{prop:component} we may assume that $G$ is connected (as $\operatorname{pw}(G) = 0$ means that $E(G) = \emptyset$). Let $(X_1,\cdots,X_t)$ be an optimal path-decomposition (of width $pw(G)$) of $G$ and such that, for every $1<i\leq t$, $|X_{i-1} \setminus X_i|\geq 1$ (note that if $X_{i-1} \subseteq X_i$, removing $X_{i-1}$ from the sequence leads to a path-decomposition with same width). The strategy for the Cop-player is to probe $X_i$ for $i =1$ to $t$ step by step. As, for $k, j$ such that $k \le i \le j$, the robber cannot move from any vertex $v \in X_j$ to a vertex $u \in X_k$ without being caught then eventually (for $i=t$), he will be located.
\end{proof}

\section{Trees}\label{sec:tree}
Following results of Seager \cite{Seager1, Seager2} one can deduce that for any tree, say $T$, $\zeta(T)$ is either 1 or 2. More precisely, she proved that one cop is sufficient to locate a robber on any tree when robber is not allowed to move to a vertex just checked by the Cop-player (in the previous round) and that this restriction is necessary for trees that contains a ternary regular tree of height $2$ as a subtree. The same bound holds true for $\zeta^*(T)$ which can be easily proven.

\begin{theorem}
If $T$ is a tree with at least 3 vertices, then $\zeta^* (T)=2$.
\end{theorem}
\begin{proof}
We will describe a simple recursive strategy for the Cop-player (see Figure \ref{drzewo}). In the first round, a pair of adjacent vertices $c_{1}$ and $c_{2}$ is chosen. 
\begin{figure}[htb]
\center
\includegraphics[width=0.7\textwidth]{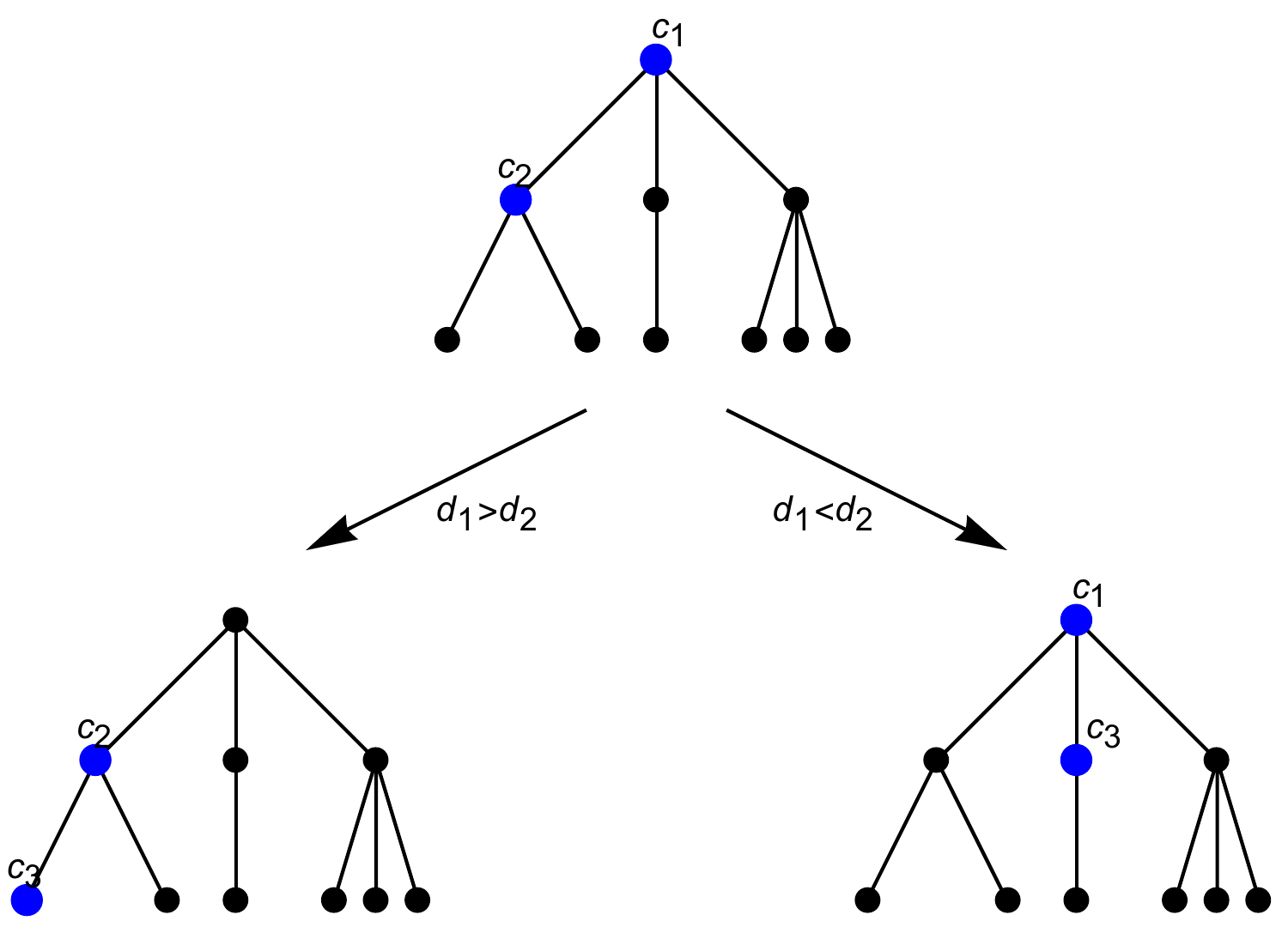}
\caption{Cop's strategy for trees.}\label{drzewo}
\end{figure}
Assuming the robber is not immediately caught, he says which of distances $d_{1}$ and $d_{2}$ from his current position to $c_1$, respectively to $c_2$,  is smaller. As each edge of the tree is a bridge (cut-edge) it splits the tree into 2 subtrees and information about $d_{1}$ and $d_{2}$ allows the Cop-player to determine the subtree in which the robber is hiding. W.l.o.g., let us assume that $d_1 > d_2 $. Then the robber is staying in the subtree containing $c_{2}$, denoted by $T_{2}$, otherwise he is staying in some subtree containing other neighbor of $c_{1}$. 

Suppose he is in $T_{2}$. Then in the second round the Cop-player will pick again $c_{2}$ and some neigbor $c_{3}$ of $c_{2}$ in $T_{2}$. Either the robber will be located immediately or after getting a response from him, Cop will know if the robber is hiding in a~subtree of $T_2$ containing $c_{3}$ or in some subtree of $T_2$ containing other neighbor of $c_{2}$. Note that the robber cannot move through $c_2$ without being located.
The process goes on, decreasing the size of the subtree where the robber can be, until the robber is caught.
\end{proof}

\section{Cartesian product of graphs}\label{sec:product}

Another result concerns the Cartesian product of graphs. Recall that the Cartesian product of graphs $G$ and $H$, is the graph $G \square H$ with vertex set $V(G \square H) = V(G) \times V(H)$ in which  $(u, v)$ is adjacent to $(u', v')$ if and only if either $\{u,u'\} \in E(G)$ and $v = v'$ or $u = u'$ and $\{v,v'\} \in E(H)$.
\begin{theorem} \label{thm:product}
For any graphs $G$ and $H$ there is
$$\zeta^*(G \square H) \le \max\{\Delta(G)+\Delta(H)+1, \Delta(G)+\zeta^*(H), \zeta^*(G)+\Delta(H)\}.$$
\end{theorem}
\begin{proof}
Let $G, H$ be two graphs. The winning strategy for the Cop-player is divided into $2$ phases. During the first phase, probing at most $\Delta(G)+\Delta(H)+1$ vertices per turn, the Cop-player will chase the robber to know the exact value of at least one coordinate of his position. Suppose that known coordinate is the one from the graph $G$ (the second case is symmetric). Then, in the second phase, probing at most $\Delta(G)+\zeta^*(H)$ vertices per turn, the Cop-player will locate the robber using the strategy for graph $H$ while always maintaining its knowledge on the $G$'s coordinate of the robber.

Details goes as follows. Let us describe the first phase. Suppose the robber is hidden at vertex $(r_G, r_H)$. Let $(u, v) \in V(G \square H)$. Probing at most $\Delta(G)+\Delta(H)+1$ vertices, the Cop-player probes each vertex in the set $N_{G\square H}[(u, v)]$. Note that $d_{G\square H}((u,v),(r_G,r_H)) <d_{G\square H}((u_G,v),(r_G,r_H))$ for all $u_G \in N_G(u)$ if and only if $u=r_G$. Symmetrically, $d_{G\square H}((u,v),$ $(r_G,r_H)) <d_{G\square H}((u,v_H),(r_G,r_H))$ for all $v_H \in N_H(v)$ if and only if $v=r_H$. Therefore, the Cop-player can recognize when she has found one of the coordinate of the robber's location and, in this case, the second phase starts. Let us show that the Cop-player eventually achieves this situation. 
Suppose $u \neq r_G$ and $v \neq r_H$. Comparing distances of vertices $(\cdot, v)$ and $(u, \cdot)$ with distance to $(u, v)$, the Cop-player may recognize the first step of a shortest path to the robber position in both graphs $G$ and $H$. Suppose $d_G(u_G, r_G) < d_G(u, r_G)$ and $d_H(v_H, r_H) < d_H(v, r_H)$ for some $u_G \in N_G(u)$ and $v_H \in N_H(v)$. Note that $$d_{G\square H}\left((u_G, v_H), (r_G, r_H)\right) = d_{G\square H}\left((u, v), (r_G, r_H)\right) -2.$$ Therefore even after the robber move the Cop-player is able to decrease the distance from the robber by choosing to probe $N_{G\square H}[(u_G, v_H)]$ in the next round.   
So, eventually, at least one coordinate of the robber position will be known by the Cop-player. 

Now, the second phase proceeds as follows. Suppose, w.l.o.g., that the known coordinate is the one from the graph $G$, say $r_G \in V(G)$, while the coordinate $r_H \in V(H)$ is unknown.  Now, the Cop-player uses the winning strategy in the graph $H$ probing $\zeta^*(H)$ vertices at each round, say $(r_G, v_1), (r_G, v_2), \dots, (r_G, v_{\zeta^*(H)})$ and at most $\Delta(G)$ vertices of the form $(u, v_1)$ for $u \in N(r_G)$. When the robber does not move on $G$-coordinate, this is simply a winning strategy. If the robber moved, then exactly one vertex of the form $(\cdot, v_1)$, say $(u, v_1)$, shows the minimal distance. Hence, the Cop-player is able to control such a move changing base $G$-coordinate to $u$ in next round. Cop-player must also check whether $H$-coordinate of some probed vertex is not equal to $H$-coordinate of the robber position which may be a part of the winning strategy on $H$\footnote{The authors are grateful to anonymous referee of pointing this out.}. Fortunately, moving on $G$ the robber does not change his $H$-coordinate, which allow Cop-player to skip one move in the strategy on $H$. Therefore, now and every second step when the robber continues moving on $G$-coordinates, if exactly one vertex of the form $(r_G, \cdot)$, say $(r_G, v)$ shows the minimal distance (which is necessary condition for $r_H = v$), then in the next round Cop-player probes (at most $\Delta(G)+\Delta(H)+1$) vertices of the closed neighborhood of vertex $(u, v)$. This guarantees to localize the robber if in the last round $r_H = v$, otherwise it allows the Cop-player to recognize where the robber has moved if he moved on $G$ once again. Otherwise, if more than one vertex has the minimal distance on $H$-coordinate, and every other second step, the Cop-player picks new vertices changing both coordinates: on $G$ according to the knowledge of the last move of the robber and on $H$ according to the winning strategy there. When the robber stops moving on $G$-coordinates the Cop-player returns to the described above strategy localizing the robber on $H$ and keeping localization on $G$.    
Therefore, as the Cop-player still goes forward with the winning strategy on $H$, the robber will eventually be located.     
\end{proof}

Mimicking the proof we get immediately an analogue for the (metric) localization game. In this case the second phase is easier as the Cop-player can simply guess the robber location when the robber moved on $G$-coordinate while Cop-player probed a vertex with $H$-coordinate of the robber's position.   
\begin{corollary}
For any graphs $G$ and $H$ there is
$$\zeta(G \square H) \le \max\{\Delta(G)+\Delta(H)+1, \Delta(G)+\zeta(H), \zeta(G)+\Delta(H)\}.$$
\end{corollary}

By a trivial induction on $d\in \mathbb{N}^*$, we also get:
\begin{corollary}
For any hypercube $Q_d= K_2 \square Q_{d-1}$, $d \geq 1$, $\zeta^*(Q_d) \leq d+1$. 
\end{corollary}

\section{Outerplanar graphs}\label{sec:outer}

In this section we prove that $\zeta (G) \leq \zeta^*(G) \leq 3$ for any outerplanar graph $G=(V,E)$. 

We first recall some basic notions of graphs. For any $S\subseteq V$, let $N(S)=\{v \in V\setminus S \mid \exists u \in S, \{u,v\} \in E\}$. A set $S$ is a separator if $G\setminus S$ has at least two connected components. A set $S$ is a {\it minimal separator} if there exist two vertices $a$ and $b$ in distinct connected components of $G \setminus S$ and no proper subset of $S$ separates $a$ and $b$. 
Any connected component $A$ of $G \setminus S$ is called a {\it full} component (with respect to $S$) if $N(A)=S$. A set $S$ is a minimal separator if and only if it has at least two full components.

\begin{definition}
	A graph is {\it outerplanar} if it has a planar embedding such that every vertex stands on the outer-face.
\end{definition}

Let us recall basic properties of outerplanar graphs.  A {\it minor} of $G$ is any graph that can be obtained from a subgraph $H$ of $G$ by contracting some edges of $H$. Recall that an outerplanar graph does not admit $K_{2,3}$ as a minor (since the class of outerplanar graphs is minor-closed and $K_{2,3}$ is not outerplanar).  
In particular, this implies that any minimal separator of an outerplanar graph has at most two vertices (since otherwise, there would be a $K_{2,3}$ minor). Moreover, for any minimal separator $S$ of an outerplanar, if $|S|=2$, then $G \setminus S$ has exactly two full connected components (otherwise there would be a $K_{2,3}$ as minor).

\begin{theorem} \label{thm:outer}
Every outerplanar graph $G$ satisfies $\zeta^* (G) \leq 3$. Moreover, if $G$ has $n$ vertices, then there is a cop strategy using $3$ probes per turn that takes $O(n^2)$ turns.
\end{theorem}

\begin{proof}
	From now on, let us assume that an outerplanar embedding of $G$ is given (this can be computed in polynomial-time) and fixed. Note that, once the embedding is fixed, it defines a cyclic ordering (that may be clockwise or counter-clockwise) of the neighbors of each vertex.
	
	The algorithm proceeds as follows, gradually reducing the set of vertices where the robber may be hidden. Initially, the Cop-player probes any vertex $v$ and knows that the robber stands at some vertex of $R=V \setminus \{v\}$ (unless the robber is immediately located). Then, after the robber's move, it can be only at some vertex of $R \cup \{v\}$. 
	
	Now, let us assume that we have reached one of the following two possible situations:
	\begin{description}
		\item[Situation $1$] There is a vertex $v$ and $R \subseteq V \setminus v$, where $R$ is a union of connected component(s) of $V \setminus \{v\}$, such that the robber stands at some vertex of $R \cup \{v\}$ and this is the turn of the Cop-player (note that we are in such situation after the first turn); or
		\item[Situation $2$] There are two vertices $\{u,v\}$ and $R \subseteq V \setminus \{u,v\}$, where $R$ is a union of connected component(s) of $V \setminus \{u,v\}$ such that the robber stands at some vertex of $R \cup \{u,v\}$ and this is the turn of the Cop-player. Moreover, there is at most one connected component of $R$ that is full with respect to $\{u,v\}$ and there is a $uv$-path in $G\setminus R$ (possibly $u$ and $v$ are adjacent). 
	\end{description}
	We present an algorithm that, starting from one such situation,
	 reaches a new such situation while strictly reducing the size of $R$. The algorithm uses $O(n)$ probes. This  ensures the localization of the robber in $O(n^2)$ turns. 
	
	First, let us consider the Situation $1$. 
	
	\begin{itemize}
		\item If $R$ is not connected, let $X$ be any connected component of $R$, the Cop-player checks whether the robber is located in $X$. For this purpose, let $\{v_1,\cdots,v_d\}$ be the neighbors of $v$ in $X$ in the order they appear in the outer-face (w.l.o.g., clockwise). Note that, by outer-planarity, for any connected component $Y$ of $X \setminus \{v_1,\cdots,v_d\}$,  $N(Y) \subseteq \{v_i,v_{i+1}\}$ for some $0\leq i < d$. 
		
		The Cop-player sequentially probes $\{v,v_i,v_{i+1}\}$ for $i$ from $1$ to $d-1$. 
		
		If, for some $1\leq i <d$, the probe at $\{v,v_i,v_{i+1}\}$ indicates that the robber is closer to $v_i$ or $v_{i+1}$ than it is to $v$, then the robber is necessarily in $X$ and we reach Situation $1$ for $v$ and $X\subsetneq R$. 
		
		Otherwise, we show that the robber cannot be in $X$. Indeed, for every $1 \leq i <d$, let $Y_i$ be the union of the connected components $Y$ of $X\setminus \{v_1,\cdots,v_d\}$ such that $N(Y) \subseteq \{v,v_1,\cdots,v_{i+1}\}$ (in particular, $N(Y_i) \subseteq \{v,v_1,\cdots,v_{i+1}\}$). Note that $Y_{d-1} = X \setminus N(v)$. By induction on $i$, let us assume that $\{v,v_j,v_{j+1}\}$ have been probed sequentially for $j$ from $1$ to $i-1$ and that the robber cannot be in $Y_{i-1}$. When $\{v,v_i,v_{i+1}\}$ is probed, if the robber occupies a connected component $Z$ of $X \setminus N(v)$ such that $N(Z) \subseteq \{v_i,v_{i+1}\}$ (i.e., in $Y_{i} \setminus Y_{i-1}$), then its distance to $v_i$ or $v_{i+1}$ should be strictly less than its distance to $v$. Moreover, if the robber was in $X \setminus Y_{i}$ before the probe, it must still be the case since $\{v,v_{i+1}\}$ separates $Y_{i}$ from $X \setminus Y_{i}$ (and also from $G \setminus X$). Hence, unless the robber is caught or detected in $Y_{i} \setminus Y_{i-1}$, it cannot be in $Y_{i}$ after this probe.
		
		Eventually, either the robber is detected in $X$ and we reach Situation $1$ (for $v$ and $X$ as described above) or we can certify that the robber cannot be in $X$, and we reach the Situation $1$ where $v$ keeps its role and $R\setminus X$ plays the role of $R$.
		
		\item Otherwise, if $R$ is connected and $v$ has a unique neighbor $w$ in $R$ (note that this is the case if and only if $\{v,w\}$ is a cut-edge/bridge), the Cop-player probes $\{v,w\}$. Then, the robber is immediately located or occupies some vertex in $R \setminus \{w\}$. We are back to Situation 1 where $w$ plays the role of $v$ and $R\setminus \{w\}$ plays the role of $R$. 
		
		\item Finally, let us assume that $R$ is connected and $v$ has $d\geq 2$ neighbors $v_1,\cdots,v_d$ in $R$ where these neighbors are ranked in the order (w.l.o.g., clockwise) they appear in the outer-face. Then, the Cop-player probes $\{v,v_1\}$. 
		Clearly, since $R$ is connected, the connected component $C$ of $R \setminus \{v_1\}$ containing $v_2$ is full with respect to $\{v,v_1\}$ (i.e., $N(C)=\{v,v_1\}$) and because $v_1$ is the first (in clockwise order) neighbor of $v$ and by outer-planarity, there is at most one full component.
Hence, we have reached Situation $2$ for the edge $\{v,v_1\}$ and $R \setminus \{v_1\}$.
		\end{itemize}

	Now, let us consider Situation $2$, i.e., it is the turn of the Cop-player and the robber is at some vertex of $R  \cup \{u,v\}$ where $R$ is a union of connected components of $V \setminus \{u,v\}$ for some vertices $u$ and $v$. Moreover, at most one connected component $F$ of $R$ is full with respect to $\{u,v\}$ (by definition of Situation $2$). In that case, the Cop-player first probes $u$ and $v$. If the robber is located (at $u$ or $v$), we are done. Otherwise, after the move of the robber, we reach exactly the same Situation $2$ with the additional information that $u$ and $v$ have been probed last.
	
	\begin{itemize}
		\item First, let us assume that there is a connected component $X$ of $R$ such that $N(X)=\{v\}$. 
		Note that during the previous probe of $u$ and $v$, the distance between $v$ and the robber must have been strictly less than the distance between $u$ and the robber, since otherwise it is already clear that the robber cannot be in $X$. In particular, it means that the distance between $u$ and the robber was at least two (during the previous probe). 
		
		Let $\{v_1,\cdots,v_d\}$ be the neighbors of $v$ in $X$ in the order (w.l.o.g., clockwise) they appear in the outer-face. The Cop-player sequentially probes $\{v,v_i,v_{i+1}\}$ (odd turns) and $\{u,v,v_{i+1}\}$ (even turns) for $i$ from $1$ to $d-1$. 
		
		The fact that $u$ and $v$ are probed every two turns prevents the robber to reach a vertex in $G \setminus R$ without being located. Indeed, at every such even turn, the distance between the robber and $v$ must be strictly less than the distance between the robber and $u$ (otherwise, it becomes sure that the robber is not in $X$, in which case we reach the Situation $2$ for $\{u,v\}$ and $R\setminus X$ plays the role of $R$). In particular, this implies that the robber is at distance at least two from $u$ and so cannot cross $u$ in two turns. 
		
		The fact that $v$ and $v_{i+1}$ are probed both during the odd and even $i${th} turns allows the proof of first case of Situation $1$ to apply (because $\{v,v_{i+1}\}$ separates $Y_i$ from $X \setminus Y_i$, where $Y_i$ is defined as in the first case of Situation $1$). 
		
		Therefore, as for the first case of Situation $1$, the robber must be in $X$ if and only if there is some odd turn $i< d$ such that the distance between $v_i$ or $v_{i+1}$ and the robber's location $r$ is less than the distance between $v$ and $r$. If the robber is in $X$, we reach the Situation $1$ for $v$ and $X\subsetneq R$ plays the role of $R$. Otherwise, we reach the Situation $2$ for $\{u,v\}$ and $R\setminus X$ plays the role of $R$.
		\item The case is symmetric (exchanging $u$ and $v$) if there is some connected component of $R$ that is adjacent only to $u$. 
		\item Hence, if neither of the above cases occurs, we may assume that $R$ is connected and that it is full with respect to $\{u,v\}$. Let $P=(v=w_0,w_1,\cdots,w_d,w_{d+1}=u)$ be the shortest path among the paths from $v$ to $u$ with internal vertices in $R$ (let us emphasize that we consider the paths with at least one vertex in $R$, even if $u$ and $v$ may be adjacent). Such a path exists because $R$ is a full component and is unique by both outer-planarity and the existence of a $uv$-path in $G\setminus R$. Note that, by outer-planarity, for any connected component $Y$ of $R \setminus P$, there is $0\leq i \leq d$ such that $N(Y) \subseteq \{w_i,w_{i+1}\}$.
		
		If there is no connected component $Y$ of $R \setminus P$ such that $N(Y)=\{v,w_1\}$, then the Cop-player probes $\{u,v,w_1\}$ and (unless the robber is located), we reach Situation $2$ for $\{w_1,u\}$ and $R \setminus \{w_1\}$. 
		Hence, let $L$ be the (unique) component of $R \setminus P$ that is full with respect to $\{v,w_1\}$. Let $R'=R \setminus (L \cup \{w_1\})$. 
		
		The Cop-player first probes $\{u,v,w_1\}$. Let $r$ be the (unknown) position of the robber during this probe.   Let us discuss the result of such a probe.
		\begin{claim}
		Either $d(w_1,r)<\min \{d(u,r),d(v,r)\}$ or Situation $2$ is reached for $\{w_1,u\}$ and $R'$ or for $\{w_1,v\}$ and $L$.
		\end{claim}
		 \begin{claimproof}
		 There are three cases to be considered.
		\begin{itemize}
			\item if $d(u,r)<d(v,r)$ and $d(u,r)\leq d(w_1,r)$, then $r$ must be in $R'=R \setminus (L \cup \{w_1\})$ and therefore, we reach Situation $2$ for $\{w_1,u\}$ and $R'$.  
			\item  if $d(v,r)<d(u,r)$ and $d(v,r)\leq d(w_1,r)$,  then $r$ must be in $L$ and therefore, we reach Situation $2$ for $\{w_1,v\}$ and $L$.
			\item if $d(v,r)=d(u,r)$, it can be checked that $w_1$ must be strictly closer to $r$ than $u$ and $v$ (by outer-planarity and by uniqueness of $P$ in $R$) unless the robber is located. 
			In particular (unless the robber is at $w_1$ and then is located), it implies that the robber is at distance at least $2$ from $u$ and $v$. 
		\end{itemize}
		\end{claimproof}
		
		Let us then assume that we are in the former case ($w_1$ closer to the robber when probing $\{u,v,w_1\}$). The aim of the following strategy is to decide whether the robber is in $L$ or in $R'=R \setminus (L \cup \{w_1\})$.  An important ingredient for the correctness of the strategy is that $u,v$ and $w_1$  will be probed every second step (even turns) which will prevent the robber to leave $R$ without being caught (since $u$ and $v$ separate $R$ from the rest of the graph). Indeed, as described above, each time $u,v$ and $w_1$ are probed, either the robber is located either in $L$ or $R'$, or it must be at distance at least $2$ from $u$ and $v$, and therefore, probing $u$ and $v$ every two steps is sufficient to avoid the robber crossing them. 
		
		W.l.o.g., let us assume that, when going clockwise along the outer-face (recall that the embedding is fixed), we first meet $v$, then $w_1$ and finally $u$. 
		Let $s_1,\cdots,s_t$ be the neighbors (in clockwise order) of $w_1$
		 in $L$. By outer-planarity, for every connected component $Y$ of $L\setminus \{s_1,\cdots,s_t\}$, there is $1 \leq i <t$ such that $N(Y)\subseteq \{s_i,s_{i+1}\}$, or $N(Y) = \{v,s_1\}$. Let $Y_0$ be the (unique if any) component of $L\setminus \{s_1,\cdots,s_t\}$ that is full with respect to $\{v,s_1\}$, and for any $1\leq i \leq t$, let $Y_i$ be the union of the connected components $Y$ of $L\setminus \{s_1,\cdots,s_t\}$ such that $N(Y)\subseteq \{v,s_1,\cdots,s_i\}$ (see Figure~\ref{fig:outerplanar}).
		 
		    \begin{figure*}[t]
  \begin{center}
   \includegraphics[width=0.55\textwidth]{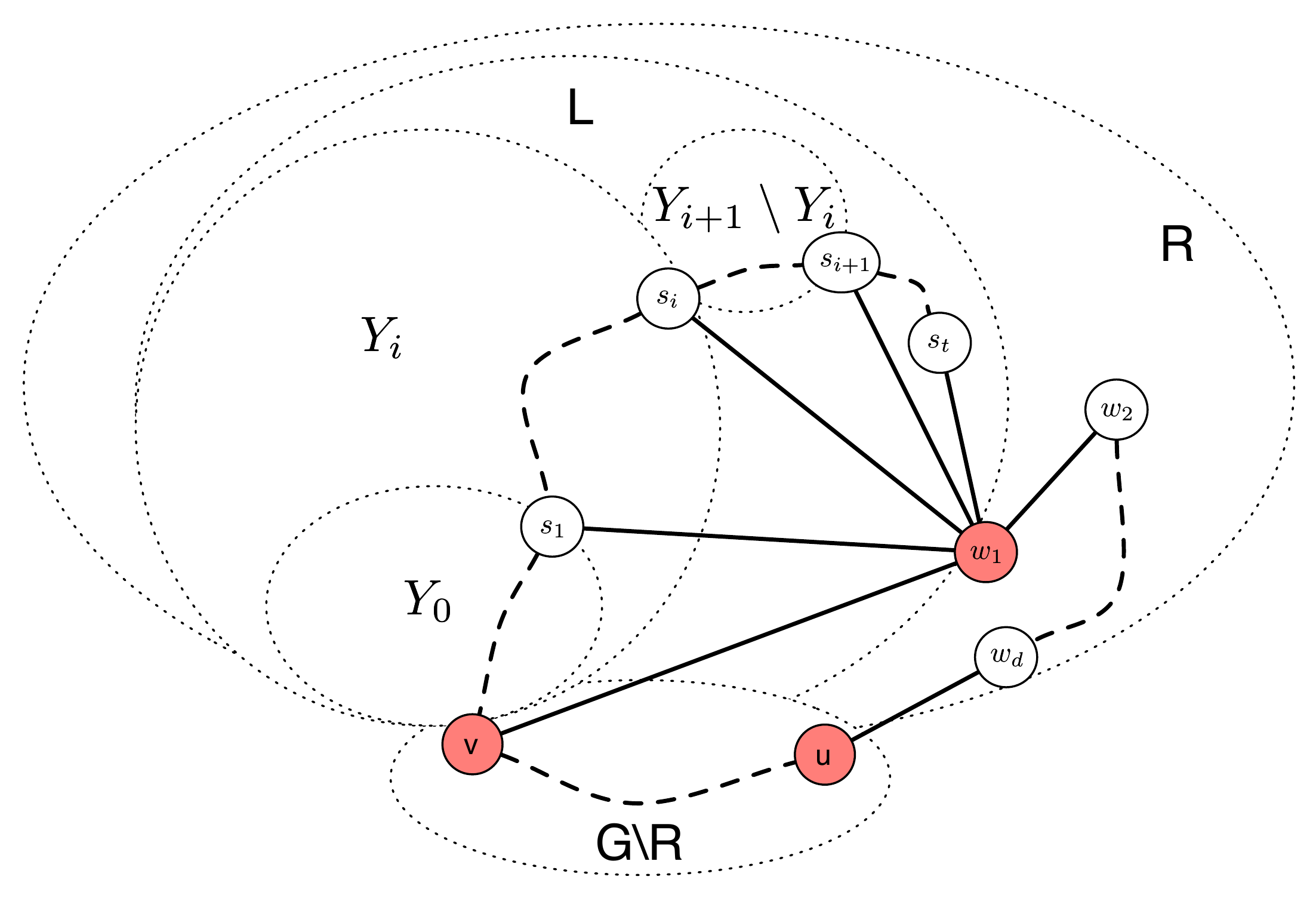}
\caption{Schematic representation of Situation $2$ in the proof of Theorem~\ref{thm:outer}. Bold lines are edges and dotted lines represent paths.}
 \label{fig:outerplanar}
\end{center}
\end{figure*}
		
		The strategy of the Cop-player proceeds as follows. First the Cop-player probes $\{v,s_1,w_1\}$ (odd turn) and then $\{u,v,w_1\}$ (even turn) and, then, for $1\leq i <t$, the Cop-player probes $\{s_i,s_{i+1},w_1\}$ (odd turns) and then $\{u,v,w_1\}$ (even turns). Note that, since $w_1$ is probed at every step, the robber can never go from $L$ to $R'$ or {\it vice versa}, i.e., if the robber is in $L$ (resp., in $R'$), he must remains in $L$ (resp., in $R'$).

		We will show the following claim which is sufficient to conclude.
		\begin{claim}
		The robber is in $L$ if and only if either $\min \{d(v,r),d(s_1,r)\} <d(w_1,r)$ (during the first odd turn), or there is an odd turn $1 \leq i <t$ such that\\ $\min \{d(s_i,r),d(s_{i+1},r)\} \leq d(w_1,r)$.
		\end{claim}

		\begin{claimproof}
		Let us start with some important observations.	
		\smallskip
		
		As already mentioned, at every even turn, $d(w_1,r)<\min\{d(u,r),d(v,r)\}$ since otherwise, the robber is located either in $L$ (in which case, we reach a Situation $2$ for $\{v,w_1\}$ and $L$) or in $R'=R \setminus (L \cup \{w_1\})$ (in which case, we reach a Situation $2$ for $\{w_1,u\}$ and $R'$). Hence, at every odd turn, $d(w_1,r) \leq d(v,r)+1$ (since the robber moves along at most one edge). 
		
		Moreover, if the robber is in $R'=R \setminus (L \cup \{w_1\})$, then at every odd turn, $d(w_1,r) \leq d(v,r)$. Indeed, let $Q$ be a shortest path from $r \in R'$ to $v$. If $Q$ does not pass through $w_1$ (otherwise it would hold that $d(v,r) > d(w_1,r)$), then it passes through $u$ and $d(v,r)>d(u,r)$. However, let $d^*(x,r)$ denote the distance between $r$ and a vertex $x$ during the previous even turn (after the probe but before the possible move of the robber), we have that $d^*(w_1,r)<d^*(u,r)$ (as mentioned above) and that $d(w_1,r)\leq d^*(w_1,r)+1$ and $d(u,r) \geq d^*(u,r)-1$ (since the robber has moved along at most one edge). Altogether $d(v,r)>d(u,r)\geq d^*(u,r)-1\geq d^*(w_1,r) \geq d(w_1,r)-1$.
		
		So, if at some odd turn $d(v,r) < d(w_1,r)$, then the robber is located in $L$ and the Situation $2$ is reached for $\{v,w_1\}$ and $L$. Let us assume that $d(v,r) \geq d(w_1,r)$ at every odd turn.

\smallskip Now, let us prove the claim by considering the odd turns one by one. 

When probing $\{v,s_1,w_1\}$, if $d(v,r)<d(w_1,r)$, then the robber is in $L$ by previous paragraph. Otherwise, if $d(s_1,r) <d(w_1,r)$, then the robber must be in $L$ (since every path from $R'$ to $s_1$ crosses $w_1$ or $v$, and $d(v,r)\geq d(w_1,r)$). On the other hand, if the robber is in $Y_1 \subseteq L$ either $d(s_1,r)<d(w_1,r)$ or $d(v,r)<d(w_1,r)$. Hence, after this probe, either the robber is located in $L$ (and we reach the Situation $2$ for $\{v,w_1\}$ and $L$) or the robber is known not to occupy a vertex in $Y_1$. 
		
		Now, let us assume by induction on $i\geq 1$ that, after the previous odd turn, it was ensured that the robber was not in $Y_i$ and that the Cop-player now probes $\{s_i,s_{i+1},w_1\}$. Note that $\{v,w_1,s_i\}$ separates $Y_i$ from the rest of the graph. 
		
		After previous odd and even turns, the robber may (moving twice) only have reached $Y_i$ by crossing $s_i$ (since if it crosses $v$ or $w_1$, it is immediately located). If such a case occurs, the robber must be at a neighbor of $s_i$ in $Y_i$ when $\{s_i,s_{i+1},w_1\}$ are probed. In that case,  $d(s_i,r)\leq d(w_1,r)<d(v,r)$ and the robber can only be in $L$. Hence,  the Situation $2$ is reached for $\{v,w_1\}$ and $L$.
		
		Now, if the robber is in $Y_{i+1} \setminus Y_i$ when $\{s_i,s_{i+1},w_1\}$ are probed, then clearly $\min\{d(s_i,r),d(s_{i+1},r)\}< d(w_1,r)$.
		
		It only remains to prove that, if $\min\{d(s_i,r),d(s_{i+1},r)\}\leq d(w_1,r)$, then the robber must be in $L$. For purpose of contradiction, let us assume that the robber is in $R'$ and $\min\{d(s_i,r),d(s_{i+1},r)\}\leq d(w_1,r)$: any shortest path from $r$ to $s_i$ or $s_{i+1}$ must pass through $v$, and hence $d(v,r)<d(r,w_1)$ which is a contradiction has shown above. 
		\end{claimproof}
		
		Therefore, after these $2(t+1)=O(n)$ rounds, either the robber has been located in $L$, in which case we are in Situation $2$ for $\{w_1,v\}$ and $L$, or the robber is known not to be in $Y_t=L$ and we reach Situation $2$ for $\{w_1,u\}$ and $R'$. 
	\end{itemize}
	In all cases, we have reached a Situation $1$ or $2$, strictly reducing the number of the possible locations of the robber.
\end{proof}

\section{Complexity}\label{sec:complexity}

In this section, we prove that the centroidal localization game (i.e., computing $\zeta^*$) is NP-hard. At first let us introduce some related properties. A set $L \subseteq V$ of vertices is called a {\it locating set} if, for every $u,v \in V \setminus L$, $N[u] \cap L \neq N[v] \cap L$. Note that a locating set must ``see" almost all vertices. Formally, for any locating set $L$, $|V \setminus N[L]| \leq 1$. Indeed, otherwise, there would be two vertices $u,v$ such that $N[u] \cap L = N[v] \cap L =\emptyset$.

In~\cite{LocalizationGame} it was proven than the localization game (i.e., computing $\zeta$) is NP-hard. At first it was proven the following
\begin{lemma}~\cite{LocalizationGame}
	Computing a minimum locating set is NP-hard in the class of graphs with diameter $2$.
\end{lemma}
Then, for any $n$-node graph $G$ with diameter $2$, a graph $G'$ was constructed by adding $n+1$ pairwise non-adjacent vertices $x_1,\cdots,x_{n+1}$, each of them being adjacent to every vertex of $V(G)$. Note that, because $G$ has diameter $2$, then $G$ is an isometric subgraph of $G'$ (i.e., distances are preserved). Finally it was proven that $\zeta(G') = k+1$, where $k$ is the minimum size of a locating set of $G$. 

Since $\zeta(G')\leq \zeta^*(G')$, it only remains to prove that $\zeta^*(G') \le k+1$. However, in~\cite{LocalizationGame}, the proof that $\zeta(G') \le k+1$ only relies on the relative order of the distances and not on the exact distances. Therefore, the proof in~\cite{LocalizationGame} actually shows that  $\zeta^*(G') \le k+1$.

It follows that:
\begin{theorem}
	The centroidal localization game is NP-hard.
\end{theorem}

\section{Euclidean plane}\label{sec:euclidean}

In this section, we study the centroidal localization game on the infinite graph with a vertex in every point of the Euclidean plane and edges between points at distance at most 1.
 In such a graph, the graph distance between a vertex and the robber is  the ceiling of their Euclidean distance. We will show that the Cop-player can locate the robber with an error of at most $2\sqrt{2}+\varepsilon$ by probing only two vertices in each round.
 
First, let us consider the possible three results of a single round. Firstly, the Cop-player may probe the vertex in which the robber stays, in which case the Cop-player wins immediately. Secondly, the Cop-player may probe two vertices such that their graph distances from the robber are the same. In this case the Cop-player could conclude that the robber is within a strap of width at most 2 (see Figure~\ref{pasek}). Last, the Cop-player may probe two vertices such that their distances to the robber differ. In this case, the Cop-player can draw a line between those two vertices (consisting of all points of the plane with equal distance from the two chosen vertices) and say that the robber must be on one side of the ``bounding" line - the side with vertex with smaller distance from the robber (see Figure \ref{strona}).
\begin{figure}[ht]
\center
\includegraphics[width=0.45\textwidth]{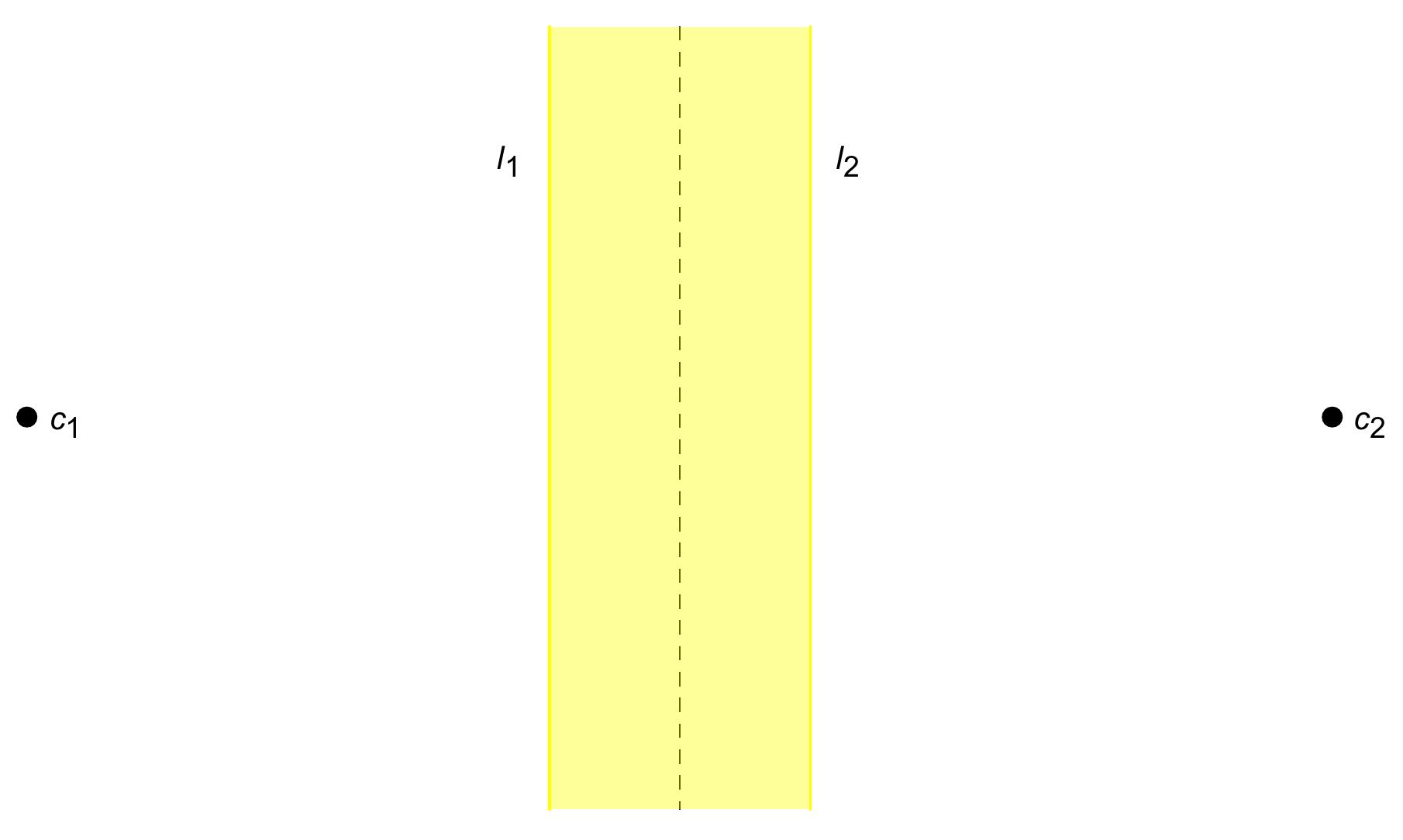}
\caption{Location of the robber when distances are the same}\label{pasek}
\end{figure}
\begin{figure}[ht]
\center
\includegraphics[width=0.45\textwidth]{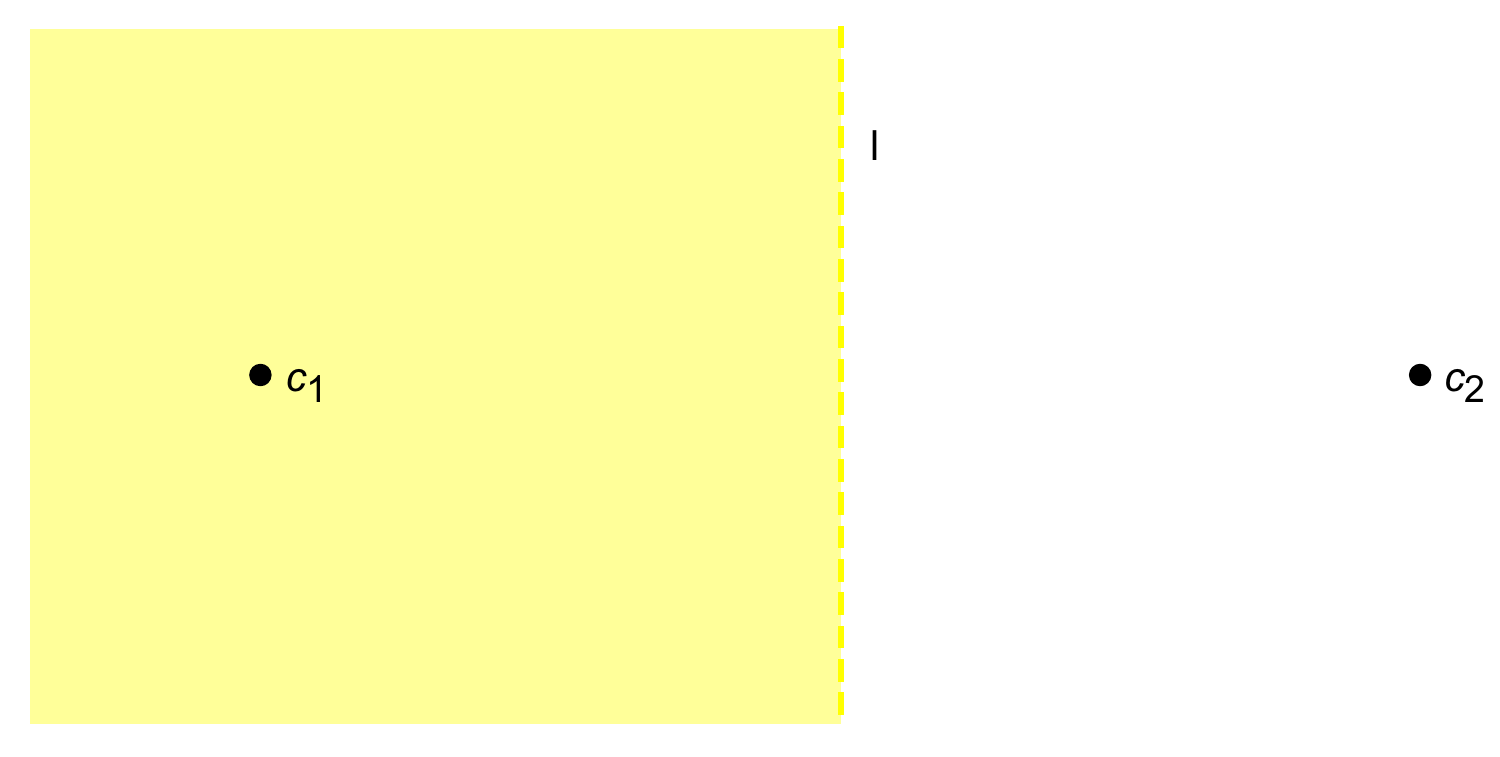}
\caption{Location of the robber when distances differ}\label{strona}
\end{figure}

\begin{lemma}
There exists a strategy allowing the Cop-player to determine a rectangle containing the robber, in a finite number of turns and probing two vertices each time.
\end{lemma}
\begin{proof}
For the convenience of the proof, let us denote the vertices chosen by the Cop-player as $c_1$ and $c_2$, and their graph distances from the robber as $d_1$ and $d_2$. Let $\Delta>1$ be any constant. We will show that, in a finite number of rounds, the Cop-player can find a rectangle, bounding the location of the robber, by first bounding his location by two vertical lines, and then by two horizontal lines.

To bound the location of the robber by two vertical lines the Cop-player will choose $c_1$ and $c_2$ from $x$ axis (let $c_1$ have smaller $x$ coordinate than $c_2$). Let us denote by $l_1$ and $l_2$ the lines bounding location of the robber. The strategy is as follows:
\begin{itemize}
\item if both $l_1$ and $l_2$ are already set, then move $l_1$ to the left by 1, and move $l_2$ to the right by 1, and end the procedure;
\item if $d_1=d_2$ then location of the robber is now bounded by the two vertical lines (the robber is within a strap of width at most 2 (see Figure~\ref{pasek}));
\item if $d_1<d_2$ then set $l_2$ to be the bounding line created by all points of equal Euclidean distance from both $c_1$ and $c_2$, and in the next round choose $c_1-(\Delta,0)$ and $c_2-(\Delta,0)$ (see Figure \ref{pl1}). In case $l_1$ is already set, shift it to the left by one;
\item if $d_1>d_2$ then set $l_1$ to be the bounding line created by all points of equal Euclidean distance from both $c_1$ and $c_2$, and in the next round choose $c_1+(\Delta,0)$ and $c_2+(\Delta,0)$. In case $l_2$ is already set, shift it to the right by one;
\end{itemize}
\begin{figure}[h]
\center
\includegraphics[width=0.6\textwidth]{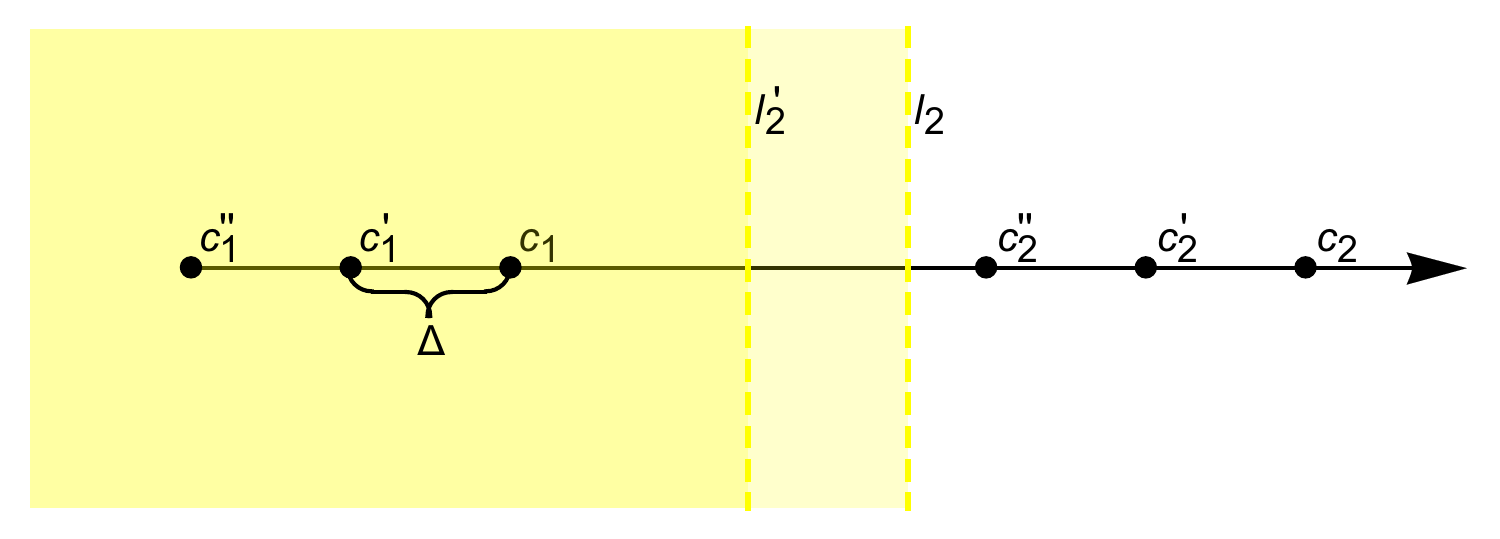}
\caption{Strategy of Cop, when $d_1<d_2$}\label{pl1}
\end{figure}
Note that this procedure will end, since $\Delta$ is strictly bigger than 1. Indeed, let us assume for purpose of contradiction that the strategy always consider the case when $d_1<d_2$ (the case $d_1>d_2$ is symmetric). Then, at each round, our procedure moves the bounding line $l_2$ to the left by $\Delta>1$, while the robber moves by at most 1 to the left. Hence, at some point, the bounding line $l_2$ must be further to the left than the robber, which means that the procedure has  reached  the case $d_1>d_2$ (and now, the second line $l_1$ will be defined). 
At the end of the procedure, we move lines $l_1$ and $l_2$ by one (first item above) since one of them bounded location of the robber in previous round.

When location of the robber is bounded by two vertical lines $l_1$ and $l_2$, then the Cop-player can repeat the bounding procedure with choosing $c_1$ and $c_2$ from $y$ axis and thus bounding the robbers location by two horizontal lines, with an extra operation of moving the two bounding vertical lines by one at each round ($l_1$ to the left and $l_2$ to the right).
\end{proof}

\begin{lemma}
For any $\varepsilon>0$, in finite number of rounds with choosing 2 vertices, the Cop-player may show a point $P$ such that the distance between the robber and $P$ is at most $3\sqrt{2}+\varepsilon$.
\end{lemma}
\begin{proof}
For the convenience of the proof, let us denote the two vertices chosen by the Cop-player as $c_1$ and $c_2$, and their distances to the robber as $d_1$ and $d_2$.
By the previous lemma, we know that, in a finite number of rounds, the Cop-player can bound the area where the robber is located by a rectangle $A_0$. After the next move the robber will be located in $A_0'=\{p\in\mathbb{R}^2:\ \exists z\in A_0\ dist(p,z)\leq 1\}$. At every round, let us denote the minimal and maximal $x$ coordinates of the bounding area by $x_1$ and $x_2$, and the minimal and maximal $y$ coordinates of the bounding area by $y_1$ and $y_2$. The strategy for the Cop-player is as follows.

Probe $c_1=(x_1,0)$ and $c_2=(x_2,0)$. If $d_1=d_2$, then the robber is in the intersection of the resulting stripe of width at most 2 and the area $A_0'$. Otherwise the Cop-player knows that the robber is either on the ``left" or ``right" side of the vertical line with $x$ coordinate equal to $(x_1+x_2)/2$. Hence, he is in one half of the previous bounding area. Let $A_1$ denote the new bounding area of the current location of the robber.
Then, in the next round, the robber is in the area $A_1'=\{p\in\mathbb{R}^2:\ \exists z\in A_1\ dist(p,z)\leq 1\}$.

In the next round choose $c_1=(0,y_1)$ and $c_2=(0,y_2)$. If $d_1=d_2$, then the robber is in the intersection of the resulting stripe of height at most 2 and the area $A_1'$. Otherwise, the Cop-player knows that the robber is either on the ``upper" or ``lower" side of a horizontal line with $y$ coordinate equal to $(y_1+y_2)/2$.  Hence, he is in one half of the previous bounding area. Let us denote by $A_2$ the bounding area of current location of the robber.
Then, in the next round, the robber is in bounding area $A_2'=\{p\in\mathbb{R}^2:\ \exists z\in A_2\ dist(p,z)\leq 1\}$.

The strategy simply repeats these two steps. Let us denote the maximal widths of the bounding areas $A_n$, $A_n'$ as $w_n, w_n'$, respectively. Notice that for even $n$ we have: $w_n'=w_n+2$, $w_{n+1}=w_n'/2$, $w_{n+1}'=w_{n+1}+2$ and $w_{n+2}=w_{n+1}'$ (unless $d_1=d_2$ in step $n+1$, in which case $w_{n+1}=2$ and $w_{n+2}=w_{n+1}'=4$). Hence $1\leq w_{n+2}\leq w_n/2 +3$. By similar analysis for odd values of $n$, we again  obtain $1\leq w_{n+2}\leq w_n/2 +3$. Analogously for all $n$ we get $1\leq h_{n+2}\leq h_n/2 +3$, where $h_n$ denotes the maximal height of $A_n$.
 Note that the limit of the sequence satisfying the recursive formula $a_{n+2}=a_n/2+3$ is $6$. Therefore, for any $\varepsilon$, there exists an integer $n$ such that $w_n\leq 6+\varepsilon$ and $h_n\leq 6+\varepsilon$. Hence, the bounding area $A_n$ is contained in a rectangle with side lengths at most $6+\varepsilon$. So, the distance between every point of $A_n$ and the point $P=(\frac{x_2-x_1}{2},\frac{y_2-y_1}{2})$ is at most $\frac{(6+\varepsilon)\sqrt{2}}{2}\leq 3\sqrt{2}+\varepsilon$.
\end{proof}

\section{Further Work}\label{sec:conclusion}

There are many interesting questions, yet to be asked, about Centroidal Coding Game. One would be to find bounds on $\zeta^*(G)$ for special classes of graphs (e.g., hypercube, partial cubes, chordal graphs, disk intersection graphs etc.). In particular: is the bound given by Theorem~\ref{thm:outer} tight?
One can notice without any effort that $\zeta(C_4) = \zeta^*(C_4) = 2$, but are there outerplanar graphs which require 3 cops?

Another problem worth to be considered is the game without the \emph{immediate catch} rule.  That is, a variant of the game where the robber could go into checked vertex unnoticed. It is easy to see that, in case of the Euclidean plane, the game does not change. For trees, let us note that the Cop-player can win such a game by probing at most $\Delta(T)$ vertices in each round (by checking all neighbors of a vertex we can tell in which subtree the robber is).

Finally, we ask what happens if the Cop-player is given a set of vertices $S\subset V(G)$ such that she can only choose vertices from this set, meaning in each round $C\subset S$. Such approach might be useful for practical applications. 

The question of the computational complexity of $\zeta^*$ in various graph classes such as bipartite graphs,  bounded treewidth graphs... is also of interest. Finally, most of the interesting turn-by-turn two-player games are known to be PSPACE-hard or even EXPTIME-complete. The exact status of the complexity of the centroidal localization game is still open.

\begin{small}
\bibliographystyle{plain}
\bibliography{CentroidalGame}

\begin{thebibliography}{10}

\bibitem{Bienstock}
D.~Bienstock.
\newblock Graph searching, path-width, tree-width and related problems (a
  survey).
\newblock {\em DIMACS Ser. in Discrete Mathematics and Theoretical Computer
  Science}, 5:33--49, 1991.

\bibitem{BonatoNowakowski}
A.~Bonato and R.~Nowakovski.
\newblock {\em The game of Cops and Robber on Graphs}.
\newblock American Math. Soc., 2011.

\bibitem{LocalizationGame}
B.~Bosek, P.~Gordinowicz, J.~Grytczuk, N.~Nisse, J.~Sok{\'{o}}l, and
  M.~Sleszynska{-}Nowak.
\newblock Localization game on geometric and planar graphs.
\newblock {\em Discrete Applied Mathematics}, 251:30--39, 2018.

\bibitem{Brandt}
A.~Brandt, J.~Diemunsch, C.~Erbes, J.~Legrand, and C.~Moffatt.
\newblock A robber locating strategy for trees.
\newblock {\em Discrete Applied Mathematics}, 232:99--106, 2017.

\bibitem{WestTCS}
J.~M. Carraher, I.~Choi, M.~Delcourt, L.~H. Erickson, and D.~B. West.
\newblock Locating a robber on a graph via distance queries.
\newblock {\em Theor. Comput. Sci.}, 463:54--61, 2012.

\bibitem{Foucaud}
F.~Foucaud, R.~Klasing, and P.~J. Slater.
\newblock Centroidal bases in graphs.
\newblock {\em Networks}, 64(2):96--108, 2014.

\bibitem{FMR+17a}
F.~Foucaud, G.~B. Mertzios, R.~Naserasr, A.~Parreau, and P.~Valicov.
\newblock Identification, location-domination and metric dimension on interval
  and permutation graphs. i. bounds.
\newblock {\em Theor. Comput. Sci.}, 668:43--58, 2017.

\bibitem{FMR+17b}
F.~Foucaud, G.~B. Mertzios, R.~Naserasr, A.~Parreau, and P.~Valicov.
\newblock Identification, location-domination and metric dimension on interval
  and permutation graphs. {II.} algorithms and complexity.
\newblock {\em Algorithmica}, 78(3):914--944, 2017.

\bibitem{Harary}
F.~Harary and R.~A. Melter.
\newblock On the metric dimension of a graph.
\newblock {\em Ars Combin.}, 2:191--195, 1976.

\bibitem{Haslegrave}
J.~Haslegrave, R.~A.~B. Johnson, and S.~Koch.
\newblock Locating a robber with multiple probes.
\newblock {\em Discrete Mathematics}, 341(1):184--193, 2018.

\bibitem{Karp}
M.~G. Karpovsky, K.~Chakrabarty, and L.~B. Levitin.
\newblock On a new class of codes for identifying vertices in graphs.
\newblock {\em {IEEE} Trans. Information Theory}, 44(2):599--611, 1998.

\bibitem{Lobstein}
A.~Lobstein.
\newblock Watching systems, identifying, locating dominating and discriminating
  codes in graphs: A bibliography.
\newblock Technical report.
\newblock Available at:
  \url{http://www.infres.enst.fr/lobstein/debutBIBidetlocdom.pdf}.

\bibitem{Seager1}
S.~M. Seager.
\newblock Locating a robber on a graph.
\newblock {\em Discrete Mathematics}, 312(22):3265--3269, 2012.

\bibitem{Seager2}
S.~M. Seager.
\newblock Locating a backtracking robber on a tree.
\newblock {\em Theor. Comput. Sci.}, 539:28--37, 2014.

\bibitem{Slater}
P.~J. Slater.
\newblock Leaves of trees.
\newblock In {\em Sixth Southeastern Conf. Combin., Graph Theory, Computing,
  Congressus}, volume~14, pages 549--559, 1975.

\bibitem{Slater1}
P.~J. Slater.
\newblock Domination and location in acyclic graphs.
\newblock {\em Networks}, 17(1):55--64, 1987.

\bibitem{Slater2}
P.~J. Slater.
\newblock Dominating and reference sets in a graph.
\newblock {\em Journal of Mathematical and Physical Sciences}, 22:445--455,
  1988.

\end{thebibliography}
\end{small}
\end{document}